\documentclass[10pt, reqno]{amsart}

\usepackage{amsmath}
\usepackage{amsfonts}
\usepackage{amssymb}
\usepackage{xcolor}
\usepackage{soul}
\usepackage{amssymb,latexsym,amsmath,amsfonts}
\usepackage{hyperref}
\usepackage{amsthm}
\usepackage{dsfont}
\usepackage[mathscr]{eucal}
\usepackage[english]{babel}
\usepackage{lipsum}
\usepackage{setspace}
\usepackage{mathtools}
\usepackage{tikz}
\usepackage{pgfplots}
\usetikzlibrary{calc,angles,positioning,intersections,quotes,decorations}

\pgfplotsset{compat=1.18}
\theoremstyle{definition}
\newtheorem*{defn}{Definition}
\newtheorem{theorem}{Theorem}[section]
\newtheorem{lemma}{ \bf Lemma}[section]
\newtheorem{proposition}{\bf Proposition}[section]
\newtheorem{corollary}{\bf Corollary}[section]
\newtheorem{conjecture}{\bf Conjecture}[section]
\theoremstyle{remark}
\newtheorem{remark}{Remark}[section]
\newcommand{\be}{\begin{equation}}
\newcommand{\ee}{\end{equation}}
\newcommand{\Bea}{\begin{eqnarray*}}
\newcommand{\Eea}{\end{eqnarray*}}
\newcommand{\bea}{\begin{eqnarray}}
\newcommand{\eea}{\end{eqnarray}}

\numberwithin{equation}{section}

\title{Some constructions of uniformly positive scalar curvature metrics on open manifolds}
\author{Anushree Das}
\date{June 2026}

\address{Beijing International Center for Mathematical Research, Peking University, Beijing, 100871, P. R. China}
\email{anushree@bicmr.pku.edu.cn}


\begin{document}

\begin{abstract}
We obtain several constructions of uniformly positive scalar curvature complete Riemannian metrics on open manifolds. For dimension $n\geq4$, we show that if such a manifold admits a proper Morse function $f$ bounded below, such that $f$ has no critical points of index $\geq n-2$, then it admits a uniformly positive scalar curvature metric. If a manifold of dimension $4 \leq n\leq 7$ admits a positive scalar curvature metric along with a compact exhaustion $\{U_i\}$ such that the boundary of each $U_i$ is minimal, then it also admits a uniformly positive scalar curvature metric. For the same dimensions, we show that if a manifold has product ends and a positive scalar curvature metric with $C$-quadratic decay at infinity for $C>4\pi^2$ with respect to some basepoint, then the existence of a certain mean convex slice far enough from the basepoint implies the existence of a uniformly positive scalar curvature metric on the manifold. In particular, manifolds which have a combination of these behaviours also admit a uniformly positive scalar curvature metric. We also study some applications of these results.
\end{abstract}

\maketitle

\section{Introduction}

A Riemannian metric $g$ on an open manifold $M$ is said to have uniformly positive scalar curvature (UPSC) if there exists a constant $\epsilon>0$ such that the scalar curvature $R_M$ satisfies $R_M\geq \epsilon$ everywhere on $M$. While the question of the existence of positive scalar curvature metrics on closed manifolds is quite well understood, the existence of positive or uniformly positive scalar curvature metrics on open manifolds has proven to be more complicated. This has been an area of significant interest, with a lot of progress having been made in recent years. 

In the article \cite{GromovBlaine} published in $1983$, Gromov and Lawson obtained several topological conditions under which an open manifold admits no complete metric of positive scalar curvature. In particular, they introduced the idea of enlargeability and $\Lambda^2$-enlargeability, and showed that a $\Lambda^2$-enlargeable manifold cannot admit a complete metric of positive scalar curvature, and neither can the product of an enlargeable manifold with $\mathbb{R}$. On the other hand, the product of an enlargeable manifold with $\mathbb{R}^2$ cannot have a complete metric of uniformly positive scalar curvature, while the product with $\mathbb{R}^3$ can. In $2010$, Chang, Weinberger and Yu proved in \cite{Taming3mfds} that $\mathbb{R}^3$ is the only contractible $3$-manifold admitting a complete UPSC metric, whereas the Whitehead manifold, for instance, does not admit one. Recently, Wang showed that the Whitehead manifold, in fact, cannot admit any complete metric of non-negative scalar curvature in \cite{WangWhitehead}. Other notable results include Chodosh and Li in \cite{ChodoshLiChao} regarding the non-existence of positive scalar curvature metrics on aspherical manifolds in dimensions $4$ and $5$ (see also Gromov in \cite{Gromov5aspherical}); non-existence of positive scalar curvature metrics in open SYS manifolds and UPSC metrics in weakly SYS manifolds in dimensions between $3$ and $7$ by Shi, Wang, Wu, and Zhu in \cite{SYS}; Orikasa's recent results regarding decay rates of positive scalar curvature metrics and non-existence of UPSC metrics in manifolds satisfying some topological constraints in \cite{Orikasa}; the proof by Chodosh, Máximo, and Mukherjee in \cite{CMM} showing that the interior of a compact, contractible $4$-manifold with boundary 
admits a UPSC metric only if it is homeomorphic to $\mathbb{R}^4$; and Sweeney Jr's related result in dimension $5$ in \cite{Sweeney}.

\begin{defn}
    Let $M$ be an open manifold of dimension $n$ equipped with a complete Riemannian metric $g$. For any point $p\in M$, we say that $(M,g)$ has at most $C$-quadratic decay of scalar curvature at infinity if there exists some $R>0$ such that for all points $x\in M$ at a distance of at least $R$ from $p$, which we denote as $d(p,x)=r_x\geq R$, the scalar curvature at $x$ satisfies
    \begin{equation}
        R_M(x)>\frac{C}{r_x^2}.
    \end{equation}
\end{defn}

While the definition of $C$-quadratic decay of curvature at infinity is dependent on the choice of basepoint $p$, it is to be noted that a different choice of basepoint results in the same decay rate for a different constant $R$. Gromov had conjectured (see \cite{FourLectures}, also \cite{Chen}) that the existence of a positive scalar curvature metric with sufficiently slow rate of decay is equivalent to the existence of a uniformly positive scalar curvature metric.

\begin{conjecture}[Gromov, \cite{FourLectures}]\label{gromovconjecture}
There exists a universal critical constant $C_n > 0$ such that the following holds. Let $M$ be an orientable $n$-manifold that admits a complete Riemannian metric of positive scalar curvature.
\begin{enumerate}
    \item For every $C < C_n$, there exists a complete Riemannian metric on $M$ of positive scalar curvature with at most $C$-quadratic decay at infinity.
    \item If $M$ admits a complete Riemannian metric with positive scalar curvature with $C$-quadratic decay at infinity for $C > C_n$, then $M$ admits a complete Riemannian metric with uniformly positive scalar curvature.
\end{enumerate}
\end{conjecture}

The conjecture is supported by the fact that $\mathbb{T}^{n-2}\times \mathbb{R}^2$, the product of the $n-2$ dimensional torus with the Euclidean plane, does not support any metric of uniformly positive scalar curvature. On the other hand, the minimum scalar curvature of a metric on this manifold is bounded above by $\frac{4\pi^2}{(r_x-R)^2}$, where $R$ and $r_x$ are as per the definition above (see \cite{Gromovmetricinequalities}). Point $2$ of Conjecture \ref{gromovconjecture} was first proven in dimension $3$ by Balacheff, de Mora Sardà, and Sabourau in \cite{Balacheff}, with decay constant $C> 64\pi^2$. Their method involved using the fill radius to show that the manifold decomposes into connected sums of closed spherical manifolds and $\mathbb{S}^1\times \mathbb{S}^2$ components, which always admit a uniformly positive scalar curvature metric.

Recently, Chen in \cite{Chen} showed the existence of the optimal constant $C_3$ in dimension $3$, thus resolving point $2$ of Conjecture \ref{gromovconjecture} in dimension $3$.

\begin{theorem}[Chen, \cite{Chen}]
    Let $M$ be a complete connected orientable Riemannian $3$-manifold. Suppose that $M$ has positive scalar curvature with at most $C$-quadratic decay at infinity for some $C > \frac{2}{3}$. Then $M$ decomposes as a possibly infinite connected sum of spherical manifolds and $\mathbb{S}^2\times\mathbb{S}^1$ summands.

The constant $\frac{2}{3}$ is optimal; the manifold $\mathbb{R}^2\times\mathbb{S}^1$ admits a complete metric of positive scalar curvature with at most $C$-quadratic decay for every $0 <C < \frac{2}{3}$.
\end{theorem}

For an open manifold $M$, we say that a nested sequence of submanifolds $\{U_i\}_{i=1}^\infty$ is a compact exhaustion of $M$ if $U_1 \subset U_2 \subset \dots \subset M$, $M = \bigcup^\infty_{i=1}U_i$, and each $U_i$ is a compact codimension-$0$ smooth submanifold with smooth boundary $\partial U_i$. 

Chen also showed that if an orientable $n$ dimensional manifold admits a complete positive scalar curvature metric with at most $C$-quadratic curvature decay at infinity, where $C>\frac{n-1}{n}$ and $3\leq n \leq 7$, then the manifold admits a compact exhaustion such that the boundaries admit positive scalar curvature metrics.

\begin{proposition}[Chen, \cite{Chen}]
Let $3 \leq n \leq 7$. Suppose $M$ is a complete orientable $n$-manifold. Suppose that $M$ has positive scalar curvature with at most $C$-quadratic decay at infinity for some $C > \frac{n-1}{n}$. Then $M$ admits a compact exhaustion $\{K_i\}_{i=1}^\infty$ such that each $\partial K_i$ admits a metric of positive scalar curvature.
\end{proposition}

Our methods are motivated by and related to the above results. Chen's methods rely on the existence of $\mu$-bubble hypersurfaces. The lower bound on the scalar curvature in the assumption implies that the $\mu$-bubble hypersurfaces satisfy a stability inequality and admit positive scalar curvature metrics. In dimension $3$, the hypersurfaces are $2$ dimensional, and hence the classification of surfaces can then be used to conclude that the hypersurfaces are in fact the sphere $\mathbb{S}^2$. This then leads to the conclusion that the manifold is a connected sum of spherical components (see \cite{JW}), and hence admits a uniformly positive scalar curvature metric. This method of constructing a $\mu$-bubble hypersurface in manifolds with some lower bound on scalar curvature has been used in dimension $3$ with considerable success. In particular, we refer to Wei, Xu, and Zhang's proof in \cite{WXZ} that a $3$-manifold with uniformly positive scalar curvature and non-negative Ricci curvature has quadratic volume growth, hence resolving a different conjecture by Gromov in the $3$ dimension case.

Our results aim at the construction of uniformly positive scalar curvature metrics on manifolds satisfying certain topological or geometric conditions, in dimensions $\geq 4$. The basic foundation of our arguments is the observation (see Lemma \ref{lemmajoiningpieces}) that if a manifold admits a decomposition into pieces such that each piece has a positive scalar curvature metric with product boundary, and the boundary metrics of the pieces are compatible with each other, then the manifold can be given a uniformly positive scalar curvature metric by joining the pieces via cylinders whose lengths are adjusted to ensure the required scalar curvature lower bounds.

Before stating our results, let us clarify our conventions. Let $M$ be a compact manifold with boundary $X$. We denote the mean curvature of the closed hypersurface $X$ by $H_X$, where $H_X$ is the trace of the second fundamental form, computed with respect to the normal vector field along $X$ which points towards the interior of $M$, unless mentioned otherwise. By our convention, the $n$ dimensional unit sphere $\mathbb{S}^n$ with the usual metric on $\mathbb{R}^{n+1}$ has positive mean curvature with respect to the unit normal vector field on $\mathbb{S}^n$ pointing to the origin.

For any point $p$ in a manifold $M$, and any $r>0$, we denote by $B(p,r)$ the closed ball of radius $r$ centred at $p$, and by $\partial B(p,r)$, we denote its boundary. We denote the interior of a set $A$ by $\mathring{A}$.

All our manifolds are assumed connected, and all the metrics we consider are complete unless specified otherwise. Our first result is the following:

\begin{theorem}\label{theoremmorsepsc}
    Let $M$ be an open, orientable, connected manifold of dimension $n\geq 4$ such that there exists a proper Morse function $f:M\to [a,\infty)$ for some $a\in \mathbb{R}$ with all critical points of $f$ having index $\leq n-3$. Then $M$ admits a complete Riemannian metric of uniformly positive scalar curvature.
\end{theorem}

Note that in dimension $3$ the only such open manifold is $\mathbb{R}^3$, which has a metric of uniformly positive scalar curvature. Our proof is based on and can be considered a corollary of Gromov and Lawson's construction of positive scalar curvature metrics on surgeries (\cite{GL}), and Walsh's results on the existence of positive scalar curvature metrics on compact cobordisms (\cite{MWalsh}).   

The next result shows that if an open and orientable manifold of dimension between $3$ and $7$ admits a complete Riemannian metric of positive scalar curvature coupled with a compact exhaustion with minimal boundaries, then it always admits a complete metric of uniformly positive scalar curvature.

\begin{theorem}\label{theorempscexhaustion}
    Let $(M,g)$ be an open, orientable, connected, and complete Riemannian manifold of dimension $4\leq n\leq 7$ with positive scalar curvature. If there exists a compact exhaustion $\{U_i\}_{i=1}^\infty$ of $M$ such that each $\partial U_i$ is a closed minimal hypersurface in $M$, then there exists a complete Riemannian metric of uniformly positive scalar curvature on $M$.
\end{theorem}

It is well known that connected sums of closed manifolds admitting positive scalar curvature metrics admit uniformly positive scalar curvature metrics. This result generalises this idea in principle to manifolds which have any positive scalar curvature metric along with an exhaustion with minimal boundaries. The argument involves showing the existence of a positive scalar curvature metric with product boundary on each $U_{i+1}\setminus \mathring{U_i}$ via the positivity of the first eigenvalues of certain functionals (see Corollary \ref{psc-conformal}), and then using the above-mentioned method to glue the pieces together to get a uniformly positive scalar curvature metric. This also shows that any positive scalar curvature metric on $\mathbb{T}^{n-2}\times\mathbb{R}^2$ cannot have a compact exhaustion with minimal hypersurface boundaries. In particular, it shows that open manifolds which do not admit any uniformly positive scalar curvature metric cannot have separating minimal hypersurfaces outside a compact set.

The next theorem provides some sufficient conditions for manifolds with $C$-quadratic decay at infinity and product ends to admit a uniformly positive scalar curvature metric. This provides a proof of the second part of Conjecture \ref{gromovconjecture} in a very particular case under additional strong hypotheses. Since the first draft of this article, the second part of the conjecture has been very recently verified in dimensions $4\textendash7$ by Wang and Zhu (\cite{conjproof}) with constant $C=\frac{n-1}{n}$. Their proof follows our construction of breaking the manifold into pieces and establishing a positive scalar curvature metric with product boundaries on those pieces, followed by joining the pieces through long cylinders to get the uniform positivity of scalar curvature. The novelty of their approach lies in showing the positivity of the respective eigenvalues in Corollary \ref{psc-conformal} in the general case, which is absent in the current article. Therefore, though the result by Wang and Zhu supersedes our next theorem, we leave it here for completeness.


\begin{theorem}\label{theoremproductupsc}
    Let $M$ be an open, orientable, and connected manifold of dimension $4\leq n\leq 7$ with product ends, equipped with a complete Riemannian metric $g$ of positive scalar curvature, that has at most $C$-quadratic decay of scalar curvature at infinity for some constant $C> 4\pi^2$. For any point $p\in M$, there exists a radius $R>0$ such that if $\partial B(p,r)$ is smooth and mean convex for some $r\geq R$, then there exists a mean convex end-separating hypersurface $\xi$ that admits a positive scalar curvature metric. If additionally the unbounded component of $M\setminus \xi$ is diffeomorphic to $\xi\times [0,\infty)$, then $M$ admits a complete Riemannian metric of uniformly positive scalar curvature.
\end{theorem}

Manifolds with product ends are interiors of compact manifolds with boundaries. In dimensions $4$ and $5$, we are aware of the non-existence of UPSC metrics in interiors of compact manifolds satisfying certain topological conditions. Here, instead of constraints on the homology groups, we rely on the existence of a mean convex hypersurface to construct a uniformly positive scalar curvature metric. We use the mean convex hypersurface to obtain a mean convex $\mu$-bubble hypersurface, and we then show that this hypersurface bounds a compact submanifold of $M$ which admits a positive scalar curvature metric with product boundary. The result follows by showing that this metric can be extended through the product region to a UPSC metric on the entire manifold. Note that our decay constant $C$ is not the optimal constant in \cite{Chen}. We refer to the remarks after the theorem for a detailed discussion on this.

Following the verification of Conjecture \ref{gromovconjecture} in dimensions $3\textendash 7$, and from Theorem \ref{theoremmorsepsc} and Theorem \ref{theorempscexhaustion}, we have the following unified theorem.

\begin{theorem}\label{theoremcombined}
    Let $M$ be an open, orientable, and connected manifold of dimension $4\leq n\leq 7$, with a metric $g$ of positive scalar curvature. Let $K\subset M$ be a compact domain such that $M\setminus \mathring{K}$ consists of unbounded connected components $\{E_i\}_{i=1}^k$ for some $k\geq 1$. Let each $E_i$ satisfy one of the following conditions:
    \begin{enumerate}
        \item $\partial E_i$ is mean convex and admits a conformal metric of positive scalar curvature, and there exists a proper Morse function $f:E_i\to [a,\infty)$ for some $a\in \mathbb{R}$ with all critical points of $f$ having index $\leq n-3$;
        \item there exists a compact exhaustion $\{U_i\}_{i=1}^\infty$ of $E_i$ such that each $\partial U_i$ is a closed minimal hypersurface in $M$;
        \item the scalar curvature $R_g$ on $E_i$ has at most $C$-quadratic decay at infinity for some $C>\frac{n-1}{n}$ with respect to some point in $M$.
    \end{enumerate}
    Then there exists a uniformly positive scalar curvature metric on $M$.
\end{theorem}

We also show some applications of our results. In particular, we show that Riemannian manifolds with a positive scalar curvature metric and a mean convex or mean concave exhaustion that admit no uniformly positive scalar curvature metric must admit a foliation near their ends by mean convex or mean concave hypersurfaces, respectively. Examples of such manifolds include manifolds which are weakly SYS but not SYS (see \cite{SYS}). We refer to Corollary \ref{corollaryfoliation} for the detailed statement and proof. Our other applications include showing that the manifolds considered above admit positive scalar curvature metrics with desired rates of volume growth, and relating the presence of a positive scalar curvature metric with mean convex stable boundary on a compact manifold to the presence of a uniformly positive scalar curvature metric on its interior.

The paper is organised as follows. In Section \ref{sectionprelim} we discuss the preliminary results and tools used in our proofs, particularly $\mu$-bubble hypersurfaces and their relevant properties. In Section \ref{sectionjoiningpieces} we prove Lemma \ref{lemmajoiningpieces}, and use it to prove Theorem \ref{theoremmorsepsc} and Theorem \ref{theorempscexhaustion}. Section \ref{sectionproductend} is dedicated to the proof of Theorem \ref{theoremproductupsc}. Finally, Section \ref{sectionapplications} contains some applications of our results, as discussed above.

\subsection*{Acknowledgement}
The author would like to thank Soma Maity for patient discussions that helped in the conception of this work, Yuguang Shi and Jian Wang for their comments, suggestions, and encouragement, and Antoine Song and Mark Walsh for helpful answers and clarifications regarding their work.

\subsection*{AI statement}
No AI was used for writing this article.

\section{Preliminaries}\label{sectionprelim}

In this section we describe the $\mu$-bubble hypersurface and discuss some of its properties that shall be crucially used in our arguments. The $\mu$-bubble hypersurfaces, as introduced by Gromov, are hypersurfaces inside compact bands that arise as minimisers of a certain functional, and can be shown to admit a metric of positive scalar curvature because of their stability. In addition to this, because they are critical points of a functional, their first variation is zero, which gives us a certain amount of control over their mean curvature. We shall exploit these properties of $\mu$-bubble hypersurfaces for our deductions. We first discuss the definition and construction of $\mu$-bubble hypersurfaces.

\begin{defn}
 
A Riemannian band $(M,g)$ is defined as a connected compact manifold $M$ with a Riemannian metric $g$ together with a decomposition of the boundary $\partial M$ as a union $\partial M = \partial^- M\cup \partial^+ M$, where both $\partial^+M$ and $\partial^- M$ are non-empty and possibly disconnected.
   
\end{defn}

Let $(M,g)$ be a Riemannian band of dimension $3\leq n\leq 7$. 
Consider a smooth function $h$ defined on the interior of $M$ such that $h \rightarrow -\infty$ at $\partial^+ M$ and $h \rightarrow +\infty$ at $\partial^- M$. Choose a Caccioppoli set $\Omega_0$ in $M$ such that the boundary of $\Omega_0$ contains $\partial^- M$, but lies away from $\partial^+ M$. That is, $\partial \Omega_0=\partial^+ \Omega_0 \cup \partial^- \Omega_0$, $\partial^ - \Omega_0 = \partial^- M$, and $\partial^+ \Omega_0\subset \mathring{M}$. Now consider the functional $\mathcal{A}(\Omega)$ where $\Omega$ is any Caccioppoli set in $M$ with $\Omega \Delta \Omega_0 \Subset \mathring{M}$ defined by the following expression:
\begin{align}\label{mudef}
    \mathcal{A}(\Omega)=\int_{\partial^* \Omega\setminus \partial M} d\mathcal{H}^{n-1} - \int_{M}(\chi_{\Omega}-\chi_{\Omega_0})hd\mathcal{H}^n.
\end{align}
Here, $\chi_{\alpha}$ is the characteristic function of any set $\alpha$, $d\mathcal{H}^{n-1}$ and $d\mathcal{H}^n$ denote the  $n-1$ and $n$ dimensional Hausdorff measures respectively, and $\partial^* \Omega$ is the reduced boundary of $\Omega$.

Gromov had conjectured the existence of a smooth minimiser of the functional $\mathcal{A}$ amongst all Caccioppoli sets on a Riemannian band, and the proof of the existence and smoothness (for dimension $3\leq n\leq 7$) of such a minimiser was carried out by Zhu in \cite{muExistence} (see also Chodosh and Li \cite{ChodoshLiChao}). Such a minimiser $\Omega'$ of $\mathcal{A}$ is called a $\mu$-bubble, and any component of the boundary $\partial^+ \Omega'$ is called a $\mu$-bubble hypersurface, which we henceforth denote by $\xi$. In addition, the $\mu$-bubble hypersurface $\xi$ (or a disjoint union of some components of $\xi$) is by construction homologous to $\partial^- M$ and is a separating hypersurface in the Riemannian band $M$. That is, it separates the ambient manifold $M$ into multiple connected components, such that after the separation no such connected component can have a component of both $\partial^+M$ and $\partial^-M$ in its boundary. For a proof of this property of the $\mu$-bubble hypersurface, we refer to Wei, Xu, and Zhang in \cite{WXZ}.

For any Caccioppoli set $\bar{\Omega}\subset M$ as described above and $\bar{\xi}=\partial^+ \bar{\Omega}$, let $\nu_{\bar{\xi}}$ be the unit normal vector field on $\bar{\xi}$ pointing outward in $\bar{\Omega}$, that is, pointing towards $\partial^+ M$.
For any smooth function $\psi$ on $\bar{\xi}$ and any vector field $V_{\psi}$
on $M$, such that $V_{\psi}$ vanishes outside a small neighborhood of $\bar{\xi}$ and agrees with $\psi \nu_{\bar{\xi}}$ on $\bar{\xi}$, denote  the flow generated by $V_{\psi}$ as $F_t$. Then, denoting  $\Omega_t=F_t(\bar{\xi})$, we get the following first variation formula of $\mathcal{A}$ at $\bar{\Omega}$:
\begin{align}\label{equationfirstvariation}
    \frac{d}{dt}_{\big|t=0}\mathcal{A}(\Omega_t)=\int_{\bar{\xi}}(H_{\bar{\xi}}-h)\psi d\mathcal{H}^{n-1}. 
\end{align}

Here $H_\xi$ is the mean curvature of $\xi$ with respect to the unit normal vector field given by $-\nu_{\bar{\xi}}$. By virtue of $\Omega'$ being a minimiser of $\mathcal{A}$, we then have that
\begin{align}
    H_\xi=h.
\end{align}

The minimality of $\Omega'$ also implies that the second variation of $\mathcal{A}$ at $\xi$ is nonnegative. Thus, we have the following formula for the $\mu$-bubble hypersurface $\xi$:

\begin{align}\label{secondvariation}
    0\leq \int_{\xi}( -\phi \Delta_{\xi}\phi + \frac{R_{\xi}}{2}\phi^2 - \frac{R_{M}}{2}\phi^2 - \frac{|A|^2}{2}\phi^2 - \frac{h^2}{2}\phi^2 - \langle\nabla_{M}h,\nu_{\xi}\rangle\phi^2) d\mathcal{H}^{n-1}.
\end{align}
Here $\phi$ is any smooth function on $\xi$, $\Delta_{\xi}$ is the Laplacian on $\xi$ with the induced metric, $R_{\xi}$ is the intrinsic scalar curvature of ${\xi}$, $R_{M}$ is the scalar curvature of $M$ restricted to $\xi$, $|A|^2$ is the squared norm of the second fundamental form of $\xi$, $\nabla_{M}h$ is the gradient of $h$, and $\nu_{\xi}$ is the outward unit normal vector of $\xi$ as described above.

The assumption that $h\to\pm \infty$ at the boundaries of $M$ can be loosened, with constraints depending on the mean curvature of $\partial^{\pm} M$. To make it explicit, we state Lemma $4.2$ by Räde in \cite{Rade}.

\begin{lemma}[Räde, \cite{Rade}]\label{lemmamububbleexistence}
  If $n \leq 7$ and $H_{\partial^\pm M} > \pm h$ on $\partial^\pm M$, there is a smooth $\mu$-bubble $\Omega'$ which is a minimiser of $\mathcal{A}$ on $M$.
\end{lemma}

In the above lemma, the mean curvatures $H_{\partial^+ M}$ and $H_{\partial^- M}$ are both taken with respect to the unit normal vector fields pointing inward towards the interior of $M$. Hence, it suffices for us to choose a function $h$ on $M$ such that $h<H_{\partial^+M}$ and $h>-H_{\partial^- M}$. If it is known a priori that $H_{\partial^+ M} > -H_{\partial^- M}$, $h$ can be chosen to be a constant function, and we get a constant mean curvature hypersurface in $M$. If $H_{\partial^+ M} > 0 > -H_{\partial^- M}$, we are thus guaranteed the existence of a minimal hypersurface in $M$. On the other hand, if nothing is known about the mean curvature of the boundary, we are forced to choose $h$ such that $h \rightarrow -\infty$ at $\partial^+ M$ and $h \rightarrow +\infty$ at $\partial^- M$, in which case we do not have any bounds on the mean curvature of the hypersurface. In this article, we shall be dealing with the situation where $\partial^- M$ is mean convex, that is, $H_{\partial^-M}>0$. Hence, for us, it shall be convenient to choose $h$ such that $h\to -\infty$ at $\partial^+ M$, and $h=0$ at $\partial^- M$.

In suitable situations, the $\mu$-bubble hypersurface $\xi$ admits a metric of positive scalar curvature conformal to the induced metric on $\xi$, following an argument similar to Schoen and Yau's construction of a positive scalar curvature metric on stable minimal hypersurfaces in \cite{SY}. Assume that $R_M\geq\epsilon >0$ on $M$. Also, assume that a function $h$ can be chosen such that

    \begin{align}\label{hsquare}
        \frac{h^2}{2}-|\nabla_{M}h|>-\frac{\epsilon}{2},
    \end{align}
and $h$ satisfies the constraints on the boundary values in Lemma \ref{lemmamububbleexistence}.
 Therefore, we can consider the functional $\mathcal{A}$ as defined in Equation \ref{mudef}, and by the existence and smoothness of the minimiser in any dimension $3\leq n \leq 7$, we get a smooth closed $\mu$-bubble hypersurface $\xi$ which is separating in $M$.

We now consider the second variation of $\xi$, which is given by Equation \ref{secondvariation}:

\begin{align*}
    0\leq \int_{\xi}( -\phi \Delta_{\xi}\phi + \frac{R_{\xi}}{2}\phi^2 - \frac{R_{M}}{2}\phi^2 - \frac{|A|^2}{2}\phi^2 - \frac{h^2}{2}\phi^2 - \langle\nabla_{M}h,\nu_{\xi}\rangle\phi^2) d\mathcal{H}^{n-1},
\end{align*}
where $\phi$ is any smooth function on $M$. This is equivalent to 
\begin{align}\label{stability2}
    \int_{\xi} (-\phi \Delta_{\xi}\phi + \frac{R_{\xi}}{2}\phi^2) d\mathcal{H}^{n-1} \geq \int_{\xi} (\frac{R_{M}}{2}\phi^2 + \frac{|A|^2}{2}\phi^2 + \frac{h^2}{2}\phi^2 + \langle\nabla_{M}h,\nu_{\xi}\rangle\phi^2) d\mathcal{H}^{n-1}.
\end{align}

From Equation \ref{hsquare} we know that $\frac{h^2}{2}\phi^2 - \langle\nabla_{M}h,\eta_{\xi}\rangle\phi^2> -\frac{\epsilon}{2}\phi^2$. Also, since $R_{M}\geq \epsilon$, therefore $\frac{R_{M}}{2}\phi^2\geq \frac{\epsilon}{2}\phi^2$. Thus we get 
\begin{align}
    \frac{R_{M}}{2}\phi^2 + \frac{h^2}{2}\phi^2 - \langle\nabla_{M}h,\nu_{\xi}\rangle\phi^2 > 0.
\end{align}
Substituting into Equation \ref{stability2}, we get
\begin{align}
    \int_{\xi} (-\phi \Delta_{\xi}\phi + \frac{R_{\xi}}{2}\phi^2) d\mathcal{H}^{n-1} > \int_{\xi} \frac{|A|^2}{2}\phi^2 d\mathcal{H}^{n-1}\geq 0.
\end{align}

Now we can follow Schoen and Yau in \cite{SY} to conclude that the operator 
\begin{align}
\mathcal{L}(\phi)=-\Delta_{\xi}\phi + \frac{n-3}{4(n-2)}R_{\xi}\phi
\end{align}
has only positive eigenvalues. Denote the first eigenvalue of $\mathcal{L}$ by $\zeta_{1,\iota^*g}$, where $\iota^* g$ indicates that the Laplacian is with respect to the induced metric on $\xi$. Then $\zeta_{1,\iota^*g}>0$, and the corresponding eigenfunction $u$ is strictly positive. Under a conformal change of the metric on $\xi$ from $\iota^*g$ to $u^{\frac{4}{n-3}}\iota^*g$, the scalar curvature of $\xi$ changes to $\frac{4(n-2)}{n-3}u^{-\frac{4}{n-3}-1}\zeta_{1,\iota^*g} u>0$. Therefore, there exists a positive function $u$ such that under the conformal change $u^{\frac{4}{n-3}}\iota^*g$, the metric on the hypersurface $\xi$ becomes one of positive scalar curvature. Thus, if it is possible to choose a function $h$ that satisfies Equation \ref{hsquare} for a positive lower bound $\epsilon$ on the scalar curvature of $M$, then the $\mu$-bubble hypersurface corresponding to that choice of $h$ admits a metric of positive scalar curvature conformal to the original induced metric.

\section{A construction of uniformly positive scalar curvature metric}\label{sectionjoiningpieces}

We begin by deriving a condition on the decomposition of an open manifold such that it ensures the existence of a uniformly positive scalar curvature metric. Then, we shall look at certain situations in which our derived condition is satisfied. 

Let $M$ be an open orientable connected manifold of dimension $n\geq 4$ with a complete Riemannian metric $g$ of positive scalar curvature. If there exists a compact exhaustion $\{U_i\}_{i=0}^\infty$ of $M$ such that for each $i$ there is some neighbourhood of $\partial U_i$ diffeomorphic to $\partial U_i\times [0,1]$ where $g$ is isometric to the product metric $g=\iota_i^*g + dt^2$, then we say $M$ admits a \textit{positive scalar curvature metric with cylindrical necks}. Here $\iota_i^*g$ is the induced metric on $\partial U_i$ from $g$. In other words, we say that $M$ has a positive scalar curvature metric with cylindrical necks if there exists a compact exhaustion of $M$ such that some positive scalar curvature metric on $M$ is a product metric in some neighbourhood of the boundary of each element of the exhaustion.

Note that while $M$ has positive scalar curvature everywhere, the scalar curvature might decay to $0$, and therefore we do not a priori get a metric of uniformly positive scalar curvature. The following lemma shows that in the case of a positive scalar curvature metric with cylindrical necks, we can indeed modify the given metric to get a metric of uniformly positive scalar curvature. 

\begin{lemma}\label{lemmajoiningpieces}
    Let $M$ be an open orientable manifold of dimension $n\geq 3$ admitting a complete Riemannian metric $g$ of positive scalar curvature with cylindrical necks. Then $M$ admits a complete metric of uniformly positive scalar curvature.
\end{lemma}

\begin{proof}
   We are given an exhaustion $\{U_i\}_{i=0}^\infty$. Define compact submanifolds $K_i$ of $M$ as $K_0=U_0$ and $K_i=U_{i}\setminus \mathring{U_{i-1}}$ for all $i>0$. So, $\partial K_0=\partial U_0$, and $\partial K_i = \partial U_i\cup \partial U_{i-1}$ for all $i>0$. For $i>0$, denote $\partial U_i$ as $\partial^+ K_i$ and $\partial U_{i-1}$ as $\partial^- K_i$ so that $\partial K_i = \partial^-K_i\cup \partial^+ K_i$. We denote the restriction of $g$ to $K_i$ by $g_i$.
   
   Each $K_i$ then has a metric $g_i$ of positive scalar curvature that is additionally a product metric in some neighbourhood of its boundary components. Let us fix some arbitrary constant $\epsilon>0$. For each $i$, we can choose constants $\lambda_i>0$ such that after scaling $g_i$ by $\lambda_i$, the scalar curvature of the new metric $\lambda_ig_i$ is at least $2\epsilon$. This is possible by the compactness of the $K_i$ and the positivity of the scalar curvature of $g_i$.

    We focus on the pieces $K_0$ and $K_1$. By definition, $g_0$ and $g_1$ are isometric in some neighbourhood of $\partial K_0$ and $\partial^- K_1$ respectively, and hence $K_0$ can be glued to $K_1$ along the boundary isometry to produce a smooth Riemannian metric. However, after scaling each $g_i$ by $\lambda_i$, this does not necessarily remain true, since we might have $\lambda_0\neq \lambda_1$. Therefore, we need to construct a metric on $K_0$ such that it agrees with the new metric on $\partial^-K_1$ in some neighbourhood of $\partial K_0$, while still maintaining control on the scalar curvature lower bounds. This would allow us to again glue $K_0$ to $K_1$ along their corresponding boundaries as before. 

    Since $\iota_0^*g$ on $\partial K_0$ and $\iota_1^*g$ on $\partial^- K_1$ are isometric, let us henceforth refer to the initial metric on $\partial K_0$ and $\partial^- K_1$ as $h$ after identification by this isometry. We need to modify the metric on $\partial K_0$ from $\lambda_0 h$ to $\lambda_1 h$. For any $T>0$, we can attach a cylinder $\partial K_0\times [0,T]$ along the boundary (identifying $\partial K_0\times\{0\}\subset \partial K_0\times [0,T]$ with $\partial K_0\subset K_0$ by the identity diffeomorphism) without changing the diffeomorphism type of $K_0$, since the attached cylinder will deformation retract to the original boundary. Consider the product metric $\overline{g}=\lambda_0(h+dt^2)$ on $\partial K_0\times[0,T]$. Since this is isometric to the new metric on $K_0$ in a neighbourhood of $\partial K_0$, the resulting metric on $K_0\cup \partial K_0\times[0,T]$ after attaching the cylinder is smooth.

    We choose a suitable constant $T$ and modify the initial product metric on $\partial K_0\times[0,T]$ so that in some neighbourhood of $\partial K_0\times\{0\}$ the metric is of the form $\lambda_0(h+dt^2)$ and in some neighbourhood of $\partial K_0\times \{T\}$ it is of the form $\lambda_1(h+dt^2)$. To this end, define a function $f:[0,T]\to\mathbb{R}^+$ as
    \begin{equation}\label{equationfunctionjoining}
        f(t)=1+(\frac{\lambda_1}{\lambda_0}-1)\frac{e^{-\frac{T}{t}}}{e^{-\frac{T}{t}}+e^{-\frac{T}{T-t}}}.
    \end{equation}
    Note that $f$ is a smooth function such that $f(0)=1$, $f(T)=\frac{\lambda_1}{\lambda_0}$, and all derivatives of $f$ vanish at $t=0$ and $t=T$. The maximum and minimum values of $f'$ and $f''$ attained on $[0,T]$ depend continuously on $T$, and as we increase $T$ both $\sup_{t\in [0,T]} \lvert f'(t)\rvert$ and $\sup_{t\in [0,T]} \lvert f''(t)\rvert$ decrease. Hence $f$ can be extended smoothly by the constant function $f=1$ for $t\leq 0$ and $f=\frac{\lambda_1}{\lambda_0}$ for $t\geq T$. See Figure \ref{figfunctionjoining} for the behaviour of such a function.

    Define the function $F:\partial K_0\times[0,T]\to\mathbb{R}^+$ as 
    \begin{equation}
        F(x,t)=f(t)^{\frac{n-2}{4}},
    \end{equation}
    that is, $F$ is defined by composing $f$ with the projection map to the second coordinate on $\partial K_0\times[0,T]$ raised to the index of $\frac{n-2}{4}$.

Recall that for a positive smooth function $u$ on a manifold $M$ of dimension $n$ with metric $g'$ and scalar curvature $R_{g'}$, the scalar curvature of the conformal metric $u^{\frac{4}{n-2}}g'$ is given by 
\begin{equation}
    R_{u^{\frac{4}{n-2}}g'}(x)=u^{-\frac{n+2}{n-2}}(uR_{g'} - \frac{4(n-1)}{n-2} \Delta u).
\end{equation}
With our choice of function $F$ on $\partial K_0\times[0,T]$, the scalar curvature after conformal change to $F^{\frac{4}{n-2}}\overline{g}$ becomes
\begin{equation}\label{conformalscalareq}
    R_{F^{\frac{4}{n-2}}\overline{g}}=F^{-\frac{n+2}{n-2}}(FR_{\overline{g}} - \frac{4(n-1)}{n-2}(\Delta_{t}F+H_t\partial_{\nu_t}F+\partial^2_{\nu_t}F) ),
\end{equation}
where $\Delta_t$ is the Laplacian operator on the hypersurface $\partial K_0\times\{t\}$, $H_t$ is the mean curvature of this hypersurface, and $\nu_t$ is a continuous choice of unit normal vector field normal to this hypersurface. By definition, $F$ is constant on each $\partial K_0\times\{t\}$. Since the metric is a product, we also have $H_t=0$. Hence, Equation \ref{conformalscalareq} becomes,
\begin{equation}\label{conformalscalareqfinal}
    R_{F^{\frac{4}{n-2}}g}=F^{-\frac{n+2}{n-2}}(FR_{\overline{g}} - \frac{4(n-1)}{n-2}\partial^2_{\nu_t}F).
\end{equation}

\begin{figure}
    \centering
    \includegraphics[width=0.5\linewidth]{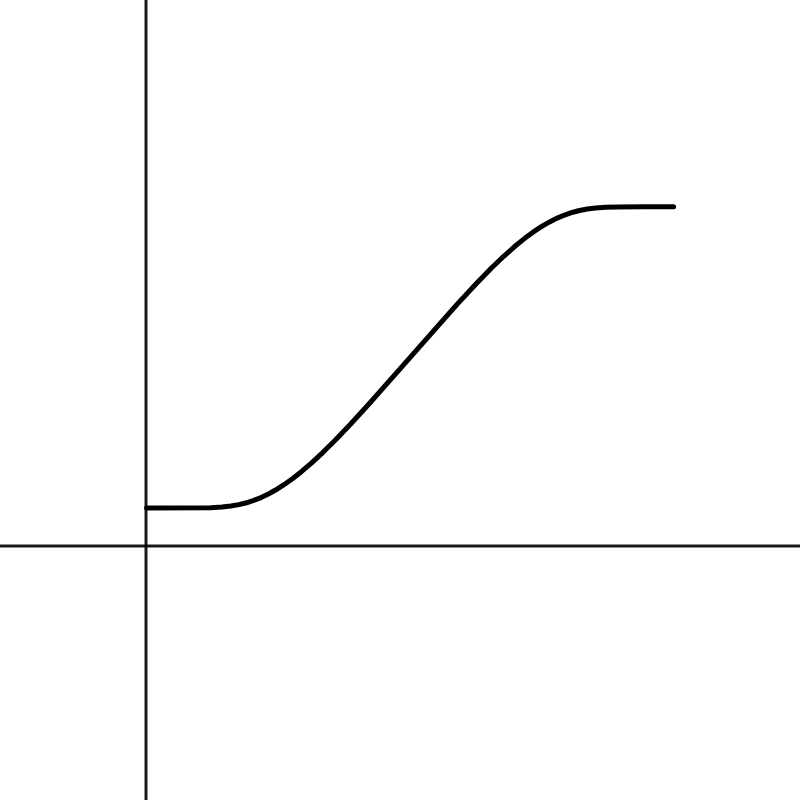}
    \caption{The function $f$ in Equation \ref{equationfunctionjoining}}
    \label{figfunctionjoining}
\end{figure}

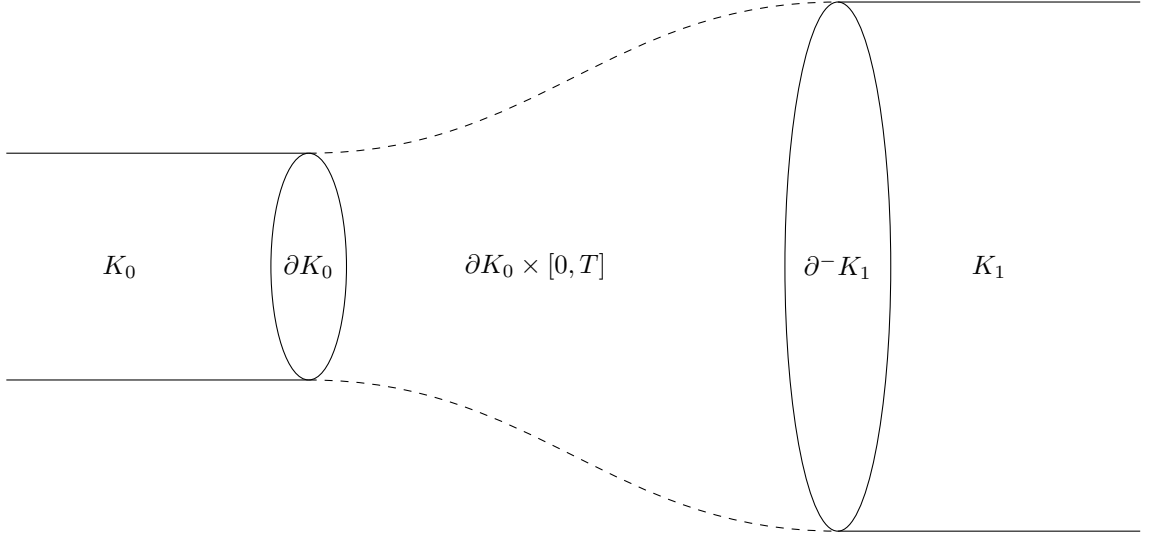
\begin{figure}[htbp]
    \centering
\begin{tikzpicture}
\draw (0,0) -- (4,0);
\draw (0,3) -- (4,3);
\draw (4,1.5) ellipse (0.5cm and 1.5 cm );
\draw (11,-2) -- (15,-2);
\draw (11,5) -- (15,5);
\draw (11,1.5) ellipse (0.7cm and 3.5 cm );
\draw[dashed] (4,3) .. controls (7,3) and (8,5) .. (11,5);
\draw[dashed] (4,0) .. controls (7,0) and (8,-2) .. (11,-2);
\node at (4,1.5) {$\partial K_0$};
\node at (11,1.5) {$\partial^- K_1$};
\node at (1.5,1.5) {$K_0$};
\node at (13,1.5) {$K_1$};
\node at (7,1.5) {$\partial K_0\times [0,T]$};
\end{tikzpicture}
\caption{Joining $K_0$ to $K_1$ through $F$}
    \label{figtikzdiagram}
\end{figure}

The first term is $F^{-\frac{4}{n-2}}R_{\overline{g}}$. At $t=0$, $F=1$, and this term is just the scalar curvature of $\overline{g}$, which is the scalar curvature of $\lambda_0 g_0$ at $\partial K_0$. At $t=T$, $F^{-\frac{4}{n-2}}=\frac{\lambda_0}{\lambda_1}$, and therefore this term becomes $\frac{\lambda_0}{\lambda_1}R_{\overline{g}}$. This is exactly the scalar curvature of $\lambda_1 g_1$ at $\partial^- K_1$. By our choice of $\lambda_i$, we therefore get that at $t=0$ and $t=T$ the first term $F^{-\frac{4}{n-2}}R_{\overline{g}}$ is at least $2\epsilon$. Since $F$ varies continuously between $1$ and $(\frac{\lambda_1}{\lambda_0})^{\frac{n-2}{4}}$ and $R_{\overline{g}}$ does not vary with $t$, we can conclude that the first term in Equation \ref{conformalscalareqfinal} is always bounded below by $2\epsilon$. 

The second term of Equation \ref{conformalscalareqfinal} is $-F^{-\frac{n+2}{n-2}}\frac{4(n-1)}{n-2}\partial^2_{\nu_t}F$. Here, note that $F$ is bounded, and as stated earlier, the absolute value of $\partial^2_{\nu_t}F$ depends on our choice of $T$. Choose $T$ large enough to ensure that 
\begin{equation}
\sup_{t\in[0,T]}\lvert F^{-\frac{n+2}{n-2}}\frac{4(n-1)}{n-2}\partial^2_{\nu_t}F \rvert\leq\epsilon.
\end{equation}
Then, we have
\begin{align*}
    R_{F^{\frac{4}{n-2}}g}=&F^{-\frac{n+2}{n-2}}(FR_{\overline{g}} - \frac{4(n-1)}{n-2}\partial^2_{\nu_t}F)\\
    &\geq 2\epsilon - \epsilon\\
    &\geq \epsilon.
\end{align*}
Thus, for our choice of $T$ and the corresponding choice of metric on $\partial K_0\times [0,T]$, we get a smooth interpolation between a product metric of the form $\lambda_0(h+dt^2)$ near $\partial K_0\times\{0\}$ and a product metric of the form $\lambda_1(h+dt^2)$ near $\partial K_0\times\{T\}$, with the scalar curvature having value at least $\epsilon$ everywhere. We now attach this cylinder to $\partial K_0\subset K_0$ along $\partial K_0\times \{0\}$ by the identity isometry as mentioned earlier, and similarly attach the other end $\partial K_0\times \{T\}$ to $\partial^- K_1$ by the identity isometry (see Figure \ref{figtikzdiagram}).

We have thus constructed a metric on $K_0\cup K_1 = U_1$ such that the scalar curvature is bounded below by $\epsilon$. Repeating this process by constructing cylinders $\partial^+ K_i\times[0,T_i]$ for suitable choices of $T_i>0$ to ensure that the scalar curvature always remains bounded below by $\epsilon$ and then attaching the cylinders to the respective boundaries results in the construction of a metric on $M = \bigcup_{i=0}^\infty K_i$ such that the scalar curvature is everywhere greater than $\epsilon$. Thus, $M$ admits a metric of uniformly positive scalar curvature. To ensure the metric is complete, take $T_i\geq 1$ for all $i$. Then, any geodesic going to an end must cross infinitely many necks of length at least $1$, and hence must have infinite length itself. Therefore, we have constructed a complete metric of uniformly positive scalar curvature on $M$. 

\end{proof}

\subsection{Uniformly positive scalar curvature metric on manifolds using Morse functions}

Having constructed a method to join pieces while retaining scalar curvature lower bounds, we now use the previous lemma to prove the existence of a uniformly positive scalar curvature metric in some classes of open manifolds. We first focus on manifolds that admit certain Morse functions. Let $M$ be an $n$ dimensional open orientable connected manifold, with $n\geq 4$, such that there is a proper Morse function on $M$ which is bounded below and has no critical points of index greater than $n-3$. We show that these manifolds always admit a metric of uniformly positive scalar curvature.

The proof relies on Gromov and Lawson's celebrated technique in \cite{GL} for the construction of a positive scalar curvature metric on manifolds obtained by a surgery of codimension at least $3$ on an ambient positive scalar curvature manifold. We shall particularly use the related result in Theorem $1.2$ by Walsh in \cite{MWalsh}, which we state here for convenience. Note that per Walsh, the term \textit{admissible Morse function} on an $n+1$ dimensional cobordism $W$ between $n$ dimensional closed manifolds $X_0$ and $X_1$ refers to a Morse function that takes constant values $0$ and $1$ on $X_0$ and $X_1$ respectively, and has no critical points of index greater than $n-2$. 

\begin{theorem}[Walsh, \cite{MWalsh}]\label{theoremwalsh}
    Let $\{W^{n+1};X_0,X_1\}$ be a smooth compact cobordism. Suppose $g_0$ is a metric of positive scalar curvature on $X_0$ and $f: W\to [0,1]$ is an admissible Morse function. Then there is a psc-metric $\overline{g}=\overline{g}(g_0,f)$ on $W$ which extends $g_0$ and has a product structure near the boundary.
\end{theorem}

Let us now prove Theorem \ref{theoremmorsepsc}.

\begin{proof}
    We are given a Morse function $f$ on $M$ which is bounded below, and hence has a global minimum. By a small modification of $f$, we can ensure that the global minimum is achieved at a single point $x\in M$. Let the global minimum be $a_0$. Then there exists some small $\delta>0$ such that $A=f^{-1}[a_0,a_0+\delta]$ is connected, smooth, and is diffeomorphic to an $n$ dimensional ball with boundary as the level set $f^{-1}(a_0+\delta)$, which is diffeomorphic to an $n-1$ dimensional sphere. Put a metric of positive scalar curvature on $A$ such that it is a product metric near the boundary. Since $A$ is diffeomorphic to a ball, such metrics exist. For instance, we can take the torpedo metric defined in \cite{MWalsh} where the metric on the boundary is the standard metric on the sphere.

    Now, consider the manifold $M\setminus \mathring{A}$ with the boundary sphere $f^{-1}(a_0+\delta)$. Choose a regular value $a_1>a_0+\delta$ and consider a level set $f^{-1}(a_1)$. The preimage $f^{-1}[a_0+\delta,a_1]$ is a submanifold with boundaries as the level sets $f^{-1}(a_0+\delta)$ and $f^{-1}(a_1)$. By reparametrising the function $f$ in this submanifold, we get a Morse function $h:f^{-1}[a_0+\delta,a_1]\to [0,1]$ such that $h(f^{-1}(a_0+\delta))=0$, $h(f^{-1}(a_1))=1$. Additionally, $h$ has critical points of the same indices as $f$, and hence all critical points of $h$ are of index $\leq n-3$. On $f^{-1}(a_0+\delta)$, we have a metric of positive scalar curvature, which is precisely the chosen metric on the sphere $\partial A$. We are now in a situation to apply Theorem \ref{theoremwalsh} to this submanifold with the function $h$. We get a metric of positive scalar curvature on $f^{-1}[a_0+\delta,a_1]$ such that it is a product metric near the boundaries. Also, since $\partial A$ and $f^{-1}(a_0+\delta)$ are now isometric by our construction, and both have product metric neighbourhoods, we can glue the submanifolds $A$ and $f^{-1}[a_0+\delta,a_1]$ by this isometry to obtain a smooth Riemannian metric on the union $f^{-1}[a_0,a_1]$.

    Choosing some regular value $a_2>a_1$ and considering the submanifold $f^{-1}[a_1,a_2]$ with the positive scalar curvature metric on $f^{-1}(a_1)$ obtained from the previous step, we again obtain a positive scalar curvature metric with product boundary after another application of Theorem \ref{theoremwalsh}. Gluing along the boundaries as before and repeating the procedure, we get a smooth Riemannian metric on $M$ of positive scalar curvature. Additionally, the metric has cylindrical necks by construction. Hence, we can now apply Lemma \ref{lemmajoiningpieces} to upgrade to a metric with uniformly positive scalar curvature. This completes the proof.
\end{proof}

\begin{remark}
    Note that Theorem \ref{theoremmorsepsc} does not require any assumption on the scalar curvature of any initial metric on the manifold. The assumptions are entirely topological, given which, we construct the uniformly positive scalar curvature metric. The existence of a positive scalar curvature metric on such manifolds is essentially a corollary of Theorem \ref{theoremwalsh} by Walsh.
\end{remark}

\subsection{Uniformly positive scalar curvature metric in manifolds with minimal exhaustions}

Let $M$ be an open manifold of dimension $4\leq n\leq 7$. Given a complete Riemannian metric of positive scalar curvature on $M$, we show that if there exists an exhaustion of $M$ with minimal hypersurface boundaries, then the metric can be upgraded to produce a metric of uniformly positive scalar curvature on $M$.

We shall crucially use Corollary 5.1 from Rosenberg, Ruberman and Xu in \cite{ConformalXu}. Before stating the corollary, let us define the terminology used. For a metric $g$ on an $n$ dimensional manifold $M'$, taking $a=\frac{4(n-1)}{n-2}$, the conformal Laplacian $\mathcal{L}$ is defined as
\begin{equation}
    \mathcal{L}(u) \coloneq -a\Delta u + R_g u,
\end{equation}
for any smooth function $u$ on $M'$.

Let $M'$ have non-empty boundary $\partial M'$. We follow the notations in \cite{ConformalXu} and denote by $\eta_{1,g}$ the first eigenvalue of the conformal Laplacian on $M'$ with respect to a Robin type boundary condition given by 
\begin{equation}\label{equationboundaryconstraint}
    \frac{\partial u}{\partial \nu} + \frac{2}{p-2}Hu = 0,
\end{equation}
where $\nu$ is the unit inward normal vector field on $\partial M'$, $H$ is the mean curvature of $\partial M'$ with respect to $\nu$, and $p=\frac{2n}{n-2}$. Note that while the convention in \cite{ConformalXu} uses the outward normal metric to define the mean curvature, the difference in sign convention in our definitions implies that the formulas remain true in our case. In particular, for both conventions, the unit sphere has positive mean curvature. When $H=0$ and $R_g>0$, it is well known that $\eta_{1,g}>0$. Recall from Section \ref{sectionprelim} that $\zeta_{1,\iota^*g}$ is the first eigenvalue of the conformal Laplacian on the boundary submanifold with respect to the induced metric $\iota^*g$.

Corollary $5.1$ in \cite{ConformalXu} states the following.
\begin{corollary}[Rosenberg, Ruberman, Xu, \cite{ConformalXu}]\label{psc-conformal}
    There exists a metric $g$ on $M$ with $\eta_{1,g}>0$, $\zeta_{1,\iota^*g}>0$ on $M$ and $\partial M$ iff there exists a psc metric $k$ on $\partial M$ which extends to a psc metric $k'$ on $M$ which is a product near $\partial M$.
\end{corollary}

Let us proceed to prove Theorem \ref{theorempscexhaustion}.

\begin{proof}
    We are given an exhaustion $\{U_i\}_{i=0}^\infty$ such that each $U_i$ is a compact submanifold of $M$ with minimal boundary. A priori, we have no information on the stability of $\partial U_i$. However, if for infinitely many $i$ we have that $\partial U_i$ is a stable minimal hypersurface, then by discarding the others we can get a new exhaustion $\{U_j\}_{j=1}^\infty$ such that each $U_j$ has a stable minimal hypersurface as its boundary. Otherwise, consider the submanifold $U_{i+1}\setminus \mathring{U_i}$. This is an $n$ dimensional manifold with boundaries $\partial U_{i}$ and $\partial U_{i+1}$, which are both minimal. Since $4\leq n\leq 7$, we can find a smooth stable minimal hypersurface $A_i$ separating the two boundaries (for instance, via the $\mu$-bubble method or by minimising area). Now, consider the compact submanifold $V_i$ of $M$ that has boundary $A_i$. By construction, $V_i$ lies in the interior of $U_{i+1}$, and therefore we can replace $U_i$ with $V_i$. Repeating for each $i$, we get an exhaustion $\{V_i\}$ with stable minimal boundary. Therefore, we can assume that each $\partial U_i$ is a stable minimal hypersurface without loss of generality. 

    Let us now consider $\partial U_i$ with the restricted metric. By our assumption, $\partial U_i$ is stable. Therefore, for any smooth function $u$ on $\partial U_i$, the stability inequality is satisfied, and we have
    \begin{equation}
        \int_{\partial U_i} (\lvert \nabla u \rvert^2 - \frac{1}{2}(R_M - R_{\partial U_i} + H^2 + \lvert A \rvert^2)u^2)d\mathcal{H}^{n-1} \geq 0.
    \end{equation}
Here $u$ is any smooth function on $\partial U_i$, $R_M$ and $R_{\partial U_i}$ are the ambient scalar curvature and the induced scalar curvature, respectively, $H$ is the mean curvature of $\partial U_i$, $A$ is its second fundamental form, and the integration on $\partial U_i$ is with respect to the measure $\mathcal{H}^{n-1}$. Since $\partial U_i$ is minimal, and $R_M>0$, we have
    \begin{equation}
        \int_{\partial U_i} (\lvert \nabla u \rvert^2 + \frac{1}{2} R_{\partial U_i} u^2)d\mathcal{H}^{n-1} \geq 0.
    \end{equation}
    
    We can now follow the method discussed in Section \ref{sectionprelim} to conclude that the first eigenvalue $\zeta_{1,\iota^* g}$ of the conformal Laplacian for $\partial U_i$ is positive.  
    In other words, there exists a conformal metric $g_i$ on $\partial U_i$ such that $g_i$ has positive scalar curvature.

    Now, note that $M=U_0 \cup \bigcup_{i=1}^\infty (U_{i}\setminus \mathring{U_{i-1}})$. Let us denote $K_i = U_{i}\setminus \mathring{U_{i-1}}$ for $i>0$, with $\partial^+ K_i=\partial U_{i}$ and $\partial^- K_i = \partial U_{i-1}$. As before, we say $K_0=U_0$. Then $K_i$ is a submanifold of $M$ with minimal stable boundary and a positive scalar curvature metric. As discussed above, $\zeta_{1,\iota^*g}>0$ on $\partial K_i$. Also, by the minimality of $\partial K_i$ and the positivity of $R_g$, we know that $\eta_{1,g}>0$. Therefore, we can apply Corollary \ref{psc-conformal} to conclude that there exists a positive scalar curvature metric $k_i$ on $K_i$ which is a product metric in some neighbourhood of $\partial K_i$. In addition, since $\partial^+ K_i = \partial^- K_{i+1}$ in the sense of being identically isometric, the new metric $\iota^* k_i$ on $\partial^+ K_i$ is isometric to $\iota^* k_{i+1}$ on $\partial^- K_{i+1}$. This follows from the proof of Corollary $5.1$ in \cite{ConformalXu}, since the new positive scalar curvature metric on the boundary is obtained by conformal change of the original metric using the first eigenfunction of the conformal Laplacian, which is identical for $\partial^+ K_i$ and $\partial^- K_{i+1}$.

    Thus, we have constructed positive scalar curvature metrics on each $K_i$ such that they are product metrics near the boundary, and the metrics agree on the corresponding boundaries. Thus, for each $i\geq0$ we can glue $\partial^+ K_i$ to $\partial^- K_{i+1}$, and the resulting metric will be smooth, and will be a product metric near the glued region. Hence, we get a smooth Riemannian metric on $M$ of positive scalar curvature with cylindrical necks. Now we can apply Lemma \ref{lemmajoiningpieces} to complete the construction of a uniformly positive scalar curvature metric on $M$.
\end{proof}

\section{Constructing positive scalar curvature metrics on interiors of compact manifolds}\label{sectionproductend}

Our next construction is that of a uniformly positive scalar curvature metric in those open manifolds that are interiors of compact manifolds with boundaries. The results in this section can be considered to be generalisations in principle of Chen's results in \cite{Chen}.

If $N$ is a compact manifold with boundary, its interior $\mathring{N}$ is known to have product ends. We say an open manifold $\mathring{N}$ of dimension $n\geq 2$ has \textit{product ends} if there exists an $n$ dimensional compact subset $U\subset \mathring{N}$ and a closed $n-1$ dimensional manifold $Z$ such that $\mathring{N}\setminus \mathring{U}$ is diffeomorphic to $Z\times [0,\infty)$. Such manifolds are also sometimes called manifolds with tame ends in the literature. Notice that by our definition, a manifold with product ends always has finitely many ends.

We can now prove Theorem \ref{theoremproductupsc}, which establishes certain conditions for the existence of a uniformly positive scalar curvature metric in a manifold $M$ with product ends of dimension $4\leq n\leq 7$.

\begin{proof}
    Fix a point $p\in M$. Since, by assumption, $M$ has quadratic curvature decay, there exists some $R_0>0$ such that outside the ball $B(p,R_0)$ the scalar curvature of $M$ decays at the rate $\frac{C}{r_x^2}$, where $r_x$ is the distance function at a point $x\in M$ from $p$. Since $M$ has product ends, we can choose a submanifold $U$ satisfying $B(p,R_0)\subset U\subset M$ such that $M\setminus \mathring{U}$ is diffeomorphic to a product $N\times [0,\infty)$ for some closed manifold $N$. Choose $R_1>0$ such that $U\subset B(p,R_1)$ and $\partial B(p,R_1)$ is homologous to $\partial U$. Define $R= 2R_1$. We shall now show that for any $r\geq R$, if $B(p,r)$ has a mean convex boundary, then $M$ admits a metric of uniformly positive scalar curvature.

    Consider the submanifold $X=B(p,r)\setminus \mathring{B(p,\frac{r}{2}})$. We assume $B(p,\frac{r}{2})$ is smooth; if not, we take a smooth hypersurface by a small deformation of $B(p,\frac{r}{2})$ in $M$ that does not decrease the distance between the boundaries of $X$. We write $\partial X = \partial^+ X \cup \partial^- X $, where $\partial^+ X=\partial B(p,\frac{r}{2})$ and $\partial^- X= \partial B(p,r)$, and we denote their respective mean curvatures with respect to the unit normal vector field pointing inwards into $X$ by $H_{\partial^+ X}$ and $H_{\partial^- X}$ respectively. Our next step involves the construction of a $\mu$-bubble hypersurface in $X$. Recall (see Section \ref{sectionprelim}) that we need to choose a smooth function $h:X\to \mathbb{R}$ that satisfies certain constraints on $\partial X$ depending on the mean curvature of the boundary. In particular, we require $H_{\partial^+X}>h$ on $\partial ^+ X$, and $H_{\partial^-X}>-h$ on $\partial^ - X$. By our assumption, we have $H_{\partial^- X}> 0$, and therefore we can choose $h$ such that $h=0$ on $\partial^- X$ and $h\rightarrow -\infty$ at $\partial^+ X$.

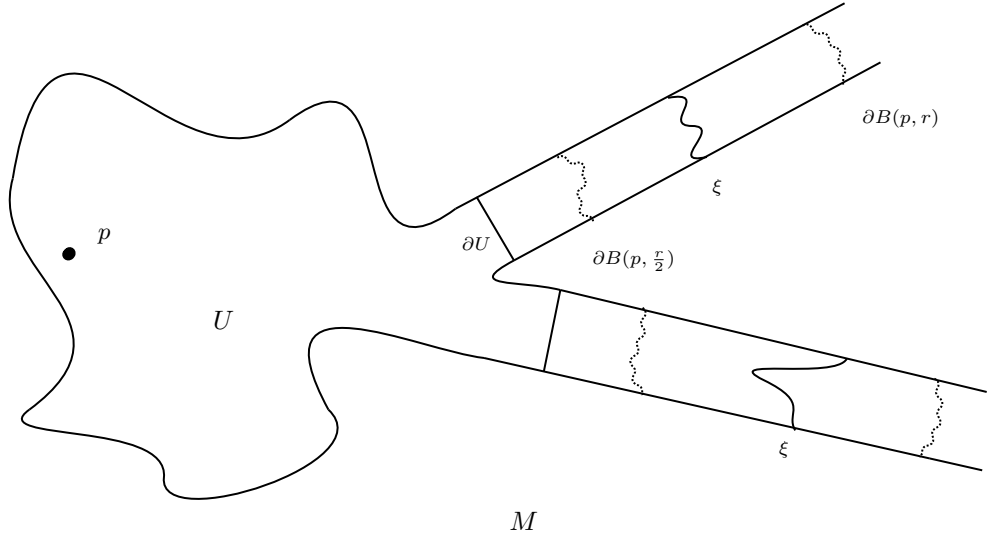
\begin{figure}
    \centering

\tikzset{every picture/.style={line width=0.75pt}}

\begin{tikzpicture}[x=0.75pt,y=0.75pt,yscale=-1,xscale=1]

\draw    (88.53,101.9) .. controls (108.79,-21.6) and (167.32,117.1) .. (227.35,72.6) ;
\draw    (227.35,72.6) .. controls (287.37,28.09) and (250.61,161.61) .. (310.63,117.1) ;
 
\draw    (96.79,217.98) .. controls (156.81,173.48) and (75.03,149.74) .. (88.53,101.9) ;
 
\draw    (96.79,217.98) .. controls (73.53,234.67) and (167.32,223.18) .. (164.32,252.1) .. controls (161.32,281.03) and (275.37,244.69) .. (246.85,217.98) ;

\draw    (246.85,217.98) .. controls (208.59,147.52) and (291.87,188.31) .. (324.89,192.02) ;

\draw    (324.89,192.02) -- (548.49,242.46) ;

\draw    (310.63,117.1) -- (505.72,14) ;
 
\draw    (339.9,143.07) -- (523.73,43.67) ;

\draw    (363.16,157.9) -- (577,209.08) ;
 
\draw    (363.16,157.9) .. controls (342.15,152.71) and (313.63,157.16) .. (339.9,143.07) ;
 
\draw    (321.14,111.17) -- (339.9,143.07) ;

\draw    (354.9,198.7) -- (363.16,157.9) ;

\draw  [dash pattern={on 0.75pt off 0.75pt}]  (361.65,91.51) .. controls (363.74,90.18) and (365.41,90.51) .. (366.64,92.48) .. controls (367.85,94.56) and (369.41,95.09) .. (371.34,94.06) .. controls (373.77,93.58) and (375.04,94.64) .. (375.13,97.25) .. controls (373.84,98.88) and (373.91,100.49) .. (375.35,102.08) .. controls (376.42,104.22) and (375.87,105.84) .. (373.7,106.94) .. controls (371.62,107.98) and (371.22,109.55) .. (372.49,111.64) .. controls (374.12,113.16) and (374.3,114.8) .. (373.01,116.56) .. controls (372.39,118.96) and (373.33,120.36) .. (375.82,120.77) .. controls (378.07,120.52) and (379.35,121.53) .. (379.66,123.78) -- (379.66,123.78) ;

\draw  [dash pattern={on 0.75pt off 0.75pt}]  (404.42,166.8) .. controls (406.31,168.4) and (406.49,170.1) .. (404.95,171.89) .. controls (403.28,173.5) and (403.21,175.19) .. (404.74,176.94) .. controls (406.11,178.82) and (405.77,180.42) .. (403.71,181.74) .. controls (401.59,182.78) and (401,184.35) .. (401.94,186.45) .. controls (402.83,188.62) and (402.22,190.15) .. (400.11,191.02) .. controls (398.03,192.36) and (397.72,194.01) .. (399.18,195.97) .. controls (400.93,197.42) and (401.14,199.04) .. (399.82,200.85) .. controls (398.79,203.02) and (399.38,204.59) .. (401.61,205.54) .. controls (403.85,206.23) and (404.66,207.7) .. (404.03,209.93) -- (404.42,210.57) ;

\draw  [dash pattern={on 0.75pt off 0.75pt}]  (484.71,25.13) .. controls (486.93,24.06) and (488.54,24.6) .. (489.55,26.77) .. controls (490.33,28.98) and (491.81,29.75) .. (493.99,29.09) .. controls (496.39,28.77) and (497.67,29.79) .. (497.83,32.14) .. controls (497.66,34.42) and (498.73,35.73) .. (501.03,36.06) .. controls (503.34,36.76) and (504.08,38.22) .. (503.24,40.45) .. controls (502.16,42.5) and (502.61,44.11) .. (504.58,45.27) .. controls (506.47,46.79) and (506.65,48.45) .. (505.1,50.24) -- (504.97,54.8) ;

\draw  [dash pattern={on 0.75pt off 0.75pt}]  (550.74,202.41) .. controls (552.9,203.48) and (553.43,205.07) .. (552.33,207.18) .. controls (551.1,209.23) and (551.44,210.87) .. (553.34,212.11) .. controls (555.17,213.71) and (555.27,215.4) .. (553.64,217.19) .. controls (551.89,218.6) and (551.74,220.22) .. (553.18,222.06) .. controls (554.43,224.17) and (554.01,225.79) .. (551.94,226.93) .. controls (549.81,227.86) and (549.16,229.42) .. (549.99,231.63) .. controls (550.67,233.89) and (549.85,235.33) .. (547.52,235.95) .. controls (545.27,236.3) and (544.33,237.63) .. (544.7,239.92) -- (543.24,241.72) ;

\draw    (548.49,242.46) -- (574.75,248.4) ;

\draw    (417.18,61.47) .. controls (438.91,57.53) and (413.43,75.57) .. (429.23,75.59) .. controls (445.02,75.61) and (415.72,95.75) .. (436.69,91.14) ;

\draw    (506.84,192.39) .. controls (501.22,203.15) and (441.31,191.48) .. (465.2,203.15) .. controls (489.09,214.82) and (475.71,219.47) .. (480.96,228.37) ;

\draw  [color={rgb, 255:red, 0; green, 0; blue, 0 }  ,draw opacity=1 ][line width=4.5] [line join = round][line cap = round] (116.3,140.1) .. controls (116.65,140.1) and (117.05,139.71) .. (117.05,139.36) ;

\draw (187.58,166.99) node [anchor=north west][inner sep=0.75pt]    {$\mathnormal{U}$};

\draw (129.8,125.71) node [anchor=north west][inner sep=0.75pt]  [font=\small]  {$\mathnormal{p}$};

\draw (312.33,129.88) node [anchor=north west][inner sep=0.75pt]  [font=\scriptsize]  {$\mathnormal{\partial } U$};

\draw (377.71,136.52) node [anchor=north west][inner sep=0.75pt]  [font=\scriptsize]  {$\partial B( p,\frac{r}{2})$};

\draw (437.82,100.64) node [anchor=north west][inner sep=0.75pt]  [font=\scriptsize]  {$\xi $};

\draw (471.21,232.3) node [anchor=north west][inner sep=0.75pt]  [font=\scriptsize]  {$\xi $};

\draw (336.15,267.87) node [anchor=north west][inner sep=0.75pt]    {$M$};

\draw (513.07,65.69) node [anchor=north west][inner sep=0.75pt]  [font=\scriptsize]  {$\partial B( p,r)$};

\end{tikzpicture}

\caption{Hypersurface $\xi$ in $M$}
    \label{figtikzdiagram2}
\end{figure}
    
    For any $\epsilon>0$, there exists a function $\rho:X\to [0,\frac{r}{2}]$ such that $\rho$ is a smoothening of the distance function $r_0$ from $\partial^-X$ on $X$, that is, for all $x\in X$ we have $\lvert \rho(x)-r_0(x)\rvert<\epsilon$,  and also the gradient $|\nabla_{X}\rho|\leq (1+\epsilon)$. We can further assume that $\rho = 0$ on $\partial^- X$ and $\rho = \frac{r}{2}$ on $\partial^+ X$. For details on the construction and existence of such a function, we refer to the appendix of \cite{WXZ}.

    Define the function $h:X\to \mathbb{R}$ as $h(x)=-\frac{2\pi(1+\epsilon)}{r}\text{tan}(\frac{\pi}{r}\rho(x))$. Then, 
    \begin{align}
    \frac{h^2}{2}-|\nabla_{X}h|\geq-\frac{2\pi^2(1+\epsilon)^2}{r^2}.
    \end{align}
    By our assumption, the scalar curvature on $M$ satisfies a quadratic decay condition with decay constant $C>4\pi^2$, and hence on $X$ the scalar curvature $R_X$ has a lower bound $\frac{C}{r^2}$. 
     Choosing $\epsilon$ small enough so that $C>4\pi^2(1+\epsilon)^2$, we get that on $X$ the scalar curvature has the lower bound 
     \begin{align}
     R_X\geq \frac{C}{r^2}> \frac{4\pi^2(1+\epsilon)^2}{r^2}.
     \end{align}
     Therefore we have
     \begin{align}\label{equationpositivityrhs}
         \frac{R_X}{2} + \frac{h^2}{2} -|\nabla_{X}h|> 0.
     \end{align}

    By our definition of $h$, we have satisfied the boundary conditions $H_{\partial^+X}>h$ on $\partial ^+ X$, and $H_{\partial^-X}>-h$ on $\partial^ - X$. Then by Lemma \ref{lemmamububbleexistence}, there exists a smooth $\mu$-bubble hypersurface $\xi$ in $X$ satisfying the second variation formula given by Equation \ref{secondvariation}:

$$
    0\leq \int_{\xi} (-\phi \Delta_{\xi}\phi + \frac{R_{\xi}}{2}\phi^2 - \frac{R_X}{2}\phi^2 - \frac{|A|^2}{2}\phi^2 - \frac{h^2}{2}\phi^2 + \langle\nabla_{X}h,\nu_{\xi}\rangle\phi^2) d\mathcal{H}^{n-1}.
$$
For convenience, we restate our notation here: $\phi$ is any smooth function on $\xi$, $\Delta_{\xi}$ is the Laplacian on $\xi$ with the induced metric, $R_{\xi}$ is the intrinsic scalar curvature of $\xi$, $|A|^2$ is the squared norm of the second fundamental form of $\xi$, and $\nu_{\xi}$ is the unit normal vector field of $\xi$ towards the direction of $\partial^- X$.

Then by Equation \ref{equationpositivityrhs}, 
\begin{align*}
    \int_{\xi} (-\phi \Delta_{\xi}\phi + \frac{R_{\xi}}{2}\phi^2) d\mathcal{H}^{n-1} & \geq \int_{\xi} (\frac{R_X}{2}\phi^2 + \frac{|A|^2}{2}\phi^2 + \frac{h^2}{2}\phi^2 - \langle\nabla_{X}h,\nu_{\xi}\rangle\phi^2) d\mathcal{H}^{n-1}\\
    & \geq 0.
\end{align*}

Now, by an argument similar to \cite{SY} explained in Section \ref{sectionprelim}, $\xi$ admits a conformal metric of positive scalar curvature. In particular, the first eigenvalue of the conformal Laplacian of $\xi$, $\zeta_{1,\iota^*g}>0$.

Now, consider the compact submanifold $V\subset M$ which has boundary $\xi$. Note that $B(p,\frac{r}{2})\subset V\subset B(p,r)$. On $V$, we will show that $\eta_{1,g}$, the first eigenvalue of the conformal Laplacian of $V$ with the boundary constraint given in Equation \ref{equationboundaryconstraint}, is also positive.

By the first variation formula stated in Equation \ref{equationfirstvariation}, the mean curvature of $\xi$ with respect to the unit normal vector field pointing towards $\partial^- X$ equals $h$ at every point of $\xi$. By our definition, $h\leq 0$ on $X$, and hence the mean curvature of $\partial V$ with respect to the unit normal vector field pointing away from $V$ is non-positive. Therefore, the mean curvature of $\partial V$ with respect to the unit normal vector field of $V$ pointing towards the interior of $V$, which we denote by $H_{\partial V}$, is non-negative. 

By Proposition 5.1 in \cite{ConformalXu}, to show that $\eta_{1,g}>0$, it is enough to show that $Y(V,\partial V, [\iota^* g])>0$, where 
\begin{align}
    Y(V,\partial V, [\iota^* g])\coloneq \inf_{u\in C^\infty, u>0}\frac{\int_V (\frac{4(n-1)}{n-2}\lvert \nabla_V u\rvert^2 + R_V u^2)d\mathcal{H}^n + 2(n-1)(n-2)\int_{\partial V}H_{\partial V}u^2d\mathcal{H}^{n-1}}{(\int_{V}u^{\frac{2n}{n-2}}d\mathcal{H}^{n})^{\frac{n-2}{n}}}.
\end{align}
 In our case, we have $R_V>c>0$ for some positive real number $c$ by the compactness of $V$, and also we know $H_{\partial V}\geq 0$. Hence, $Y(V,\partial V, [\iota^* g])>0$ and thus $\eta_{1,g}>0$. Now we are in a position to apply Corollary \ref{psc-conformal} and conclude that there exists a positive scalar curvature metric $g'$ on $V$ with product boundary.

 Now, if we additionally know that $M\setminus \mathring{V}$ is diffeomorphic to $\partial V\times [0,\infty)$, we can write $M$ as $V\cup (\partial V\times [0,\infty))$. The metric on $\partial V$ induced by $g'$, which we denote as $\iota^* g'$, is a positive scalar curvature metric. On $\partial V\times [0,\infty)$ consider the product metric with $\iota^*g'$. Since $g'$ itself is a product metric near $\partial V$, we can identify $\partial V\subset V$ with $\partial V\times \{0\}\subset \partial V\times [0,\infty)$ to get a smooth Riemannian metric on $M$. This is a metric of uniformly positive scalar curvature, since the scalar curvature is bounded below by the scalar curvature on the compact submanifold $V$. This completes the proof.
\end{proof}

Note that the mean convexity of $B(p,r)$ is crucial to the proof. We use it to prove the existence of a $\mu$-bubble hypersurface which is mean convex, which is then used to apply Corollary \ref{psc-conformal} and conclude that there exists a positive scalar curvature metric on $V$ with product boundary. Without the mean convexity of $\partial V$, it is difficult to ensure the positivity of $\eta_{1,g}$. Hence, the crucial element for the construction of a uniformly positive scalar curvature metric by this method is the existence of a stable mean convex hypersurface in the product region.

 \begin{remark}
Instead of the assumption of mean convexity of some geodesic ball, the same proof works with the decay constant lower bound of $16\pi^2$ if we instead assume the existence of any mean convex separating hypersurface in an annulus $B(p,2r)\setminus \mathring{B(p,r)}$.
\end{remark}

\begin{remark}
The decay constant in Theorem \ref{theoremproductupsc} is likely not optimal. We can retrieve the constant $\frac{n-1}{n}$ conjectured by Chen in \cite{Chen} if we instead assume the mean convexity of the geodesic sphere at a particular distance. In particular,  for the function $f$ defined in Lemma $3.4$ of \cite{Chen}, if $\partial B(p, R')$ is mean convex for a large enough $R'\in f^{-1}(0)$, then the same construction as above applies with decay constant $C>\frac{n-1}{n}$.
\end{remark}

\begin{remark}
We use the product structure of the manifold since we assume the end of the manifold is diffeomorphic to $\partial V\times[0,\infty)$. If $M\setminus \mathring{V}$ is instead assumed to admit a proper Morse function that is constant on $\partial V$ and has no critical points of index $\geq n-2$, then by the construction described in Theorem \ref{theoremmorsepsc}, we can similarly construct a metric of uniformly positive scalar curvature on $M$.
\end{remark}

\subsection*{Proof of Theorem \ref{theoremcombined}}

Let $K\subset M$ be such that each $E_i\subset M\setminus \mathring{K}$ satisfies one of the conditions in Theorem \ref{theoremcombined}. Then we can take a set $K\subset K'\subset M$ such that $K'$ has a metric of positive scalar curvature with product boundaries using the constructions discussed above. The boundary of $\partial K'$ will consist of $\partial E_i$ on those $E_i$ which satisfy condtion $(1)$, and stable separating minimal hypersurfaces and separating $\mu$-bubble minimisers (as per \cite{conjproof}) in $E_i$ for those $E_i$ which satisfy conditions $(2)$ and $(3)$ respectively. On each $E_i\cap (M\setminus \mathring{K'})$, we get a uniformly positive scalar curvature metric with product boundary, following Theorem \ref{theoremmorsepsc}, Theorem \ref{theorempscexhaustion} and the construction in \cite{conjproof}. As before, the $E_i\cap (M\setminus \mathring{K'})$ can be joined smoothly to $K'$ via the respective boundaries to get a complete metric of uniformly positive scalar curvature on $M$.

\section{Applications}\label{sectionapplications}

We discuss a few straightforward applications of our methods in this section. The first of those concerns the existence of mean convex or mean concave foliations in positive scalar curvature manifolds, and closely relates to the methods by Song in \cite{Song}. We refer to Appendix $C$ of \cite{Song} for a short discussion on mean curvature flow or level set flow, and to Section $2$ of the same for the definition of singular strictly mean convex foliation. We begin with the following corollary, which is a direct application of Theorem \ref{theorempscexhaustion}.

\begin{corollary}
    Let $(M,g)$ be an open, orientable, connected, and complete Riemannian manifold of dimension $4\leq n\leq 7$ with positive scalar curvature such that $M$ admits no metric of uniformly positive scalar curvature. Then for any $p\in M$ there exists some $R>0$ such that $M$ does not contain a compact minimal hypersurface lying in $M\setminus \mathring{B(p,R)}$ and separating the ends from a compact domain containing $B(p,R)$. In particular, there exists no compact exhaustion of $M$ with minimal hypersurface boundaries.
\end{corollary}

\begin{proof}
    If on the contrary $M$ had a compact minimal hypersurface lying in $M\setminus \mathring{B(p,r)}$ and bounding a domain containing $B(p,r)$ for all $r$, then taking a sequence $r_1<r_2<r_3\dots$ and choosing corresponding minimal hypersurfaces $Z_i$ such that $Z_i\subset M\setminus \mathring{B(p,r_i)}$ and $Z_i\subset B(p,r_{i+1})$ would result in an exhaustion of $M$ by compact domains with minimal hypersurface boundaries. However, by Theorem \ref{theorempscexhaustion}, the existence of such an exhaustion coupled with a positive scalar curvature metric on $M$ implies the existence of a uniformly positive scalar curvature metric on $M$, which contradicts our assumption.
\end{proof}

\begin{corollary}\label{corollaryfoliation}
    Let $(M,g)$ be an open, orientable, and connected complete Riemannian manifold of dimension $4\leq n\leq 7$ with positive scalar curvature such that $M$ admits no metric of uniformly positive scalar curvature. If there exists a compact exhaustion $\{U_i\}$ of $M$ such that each $\partial U_i$ is mean convex, then one of the following holds.
    \begin{enumerate}
        \item There exists a compact subset $U\subset M$ such that for each $r>0$ there is a stable minimal hypersurface intersecting $U$ and having points lying at distance greater than $r$ from $U$;
    \item There exists a (possibly empty) compact subset $U\subset M$ having the property that there is a (possibly empty) finite collection of minimal closed hypersurfaces $\{\xi_j\}\in U_i\setminus U$ bounding compact domains $\{K_{j'}\}$ in $U_i\setminus U$ such that $U_i\setminus (U\bigcup_{j'} K_{j'})$ admits a singular strictly mean convex foliation whenever it is non-empty. 
    \end{enumerate}
If instead there exists an exhaustion $\{U_i\}$ where each $\partial U_i$ is mean concave, then there exists $i_0\geq0$ and a (possibly empty) collection of minimal closed hypersurfaces $\{\xi_j\}\in M\setminus U_i$ bounding compact domains $\{K_{j'}\}$ in $M\setminus U_i$ such that $M\setminus (U_i\bigcup_{j'} K_{j'})$ admits a singular strictly mean concave foliation for all $i\geq i_0$.

\end{corollary}

\begin{proof}
    Let $U_0\subset U_1\subset U_2\dots$ be a compact exhaustion of $M$ such that each boundary component of $U_i$ is mean convex. We start with $\partial U_0$. On running the mean curvature flow from $\partial U_0$, and by the mean convexity of $\partial U_0$, we get one of two possible situations:
    \begin{enumerate}
        \item either the mean curvature flow converges to a minimal hypersurface $Z_0$ in $U_0$;
        \item or we get a singular strictly mean convex foliation of $U_0$, where the leaves are the level sets of the flow.
    \end{enumerate}
    Note that in the first case, $Z_0$ is homologous to $\partial U_0$. Each component of $Z_0$ bounds a compact domain. Denote by $V_0$ the union of the compact submanifolds of $M$ which have $Z_0$ as their boundary. Then $V_0\subset U_0$. Now, in either case, we proceed to the next element of the exhaustion, $U_1$. We again run a mean curvature flow starting at $\partial U_1$, and because of convexity, the level sets of the flow again move towards the interior of $U_1$. If we found a minimal hypersurface $Z_0$ in the first iteration, then the mean curvature flow cannot cross $Z_0$ by the maximum principle, and hence we again get one of two possibilities:
    \begin{enumerate}
        \item either the flow stops at some minimal hypersurface $Z_1$ in $U_1\setminus V_0$;
        \item or we have a singular weakly mean convex foliation of $U_1\setminus V_0$.
    \end{enumerate}

    The components of the new minimal hypersurface $Z_1$ must together bound compact domains in $U_1\setminus V_0$. Take those components that do not separate a core in $U_0$ from the end and denote the compact domain bounded by them as $\{K_{j'}\}$. If we have new minimal hypersurface components that separates the core from the end, denote the domain bounded by the union of these components as $V_1$. By the flow, we have a strictly mean convex foliation of $U_1\setminus (V_1\cup_{j'} K_{j'})$. If no new minimal hypersurface separates the core from the end, then we must have $V_1=V_0$.
    
    
    We continue iteratively, running a mean curvature flow starting at $\partial U_i$ for each $i$. At each stage, we either get a minimal hypersurface $Z_i$ homologous to $\partial U_i$, or we get a singular strictly mean convex foliation of either $U_i$ or of $U_i\setminus (V_j\cup_{j'}K_{j'})$ for some $j<i$. Here by $V_j$ we denote the compact submanifold of $M$ containing the core with components of $Z_j$ as boundary. Note that again by the maximal principle, a minimal hypersurface $Z_i$ cannot lie inside $U_{i-1}$, since otherwise it would be equal to $Z_{i-1}$. Similarly, two minimal hypersurfaces cannot intersect in this case without being equal. Hence, the minimal hypersurface components of $\{Z_i\}$ which separate the core are nested.

    Assume that we get infinitely many core separating minimal hypersurfaces $Z_i$. Since the $\{U_i\}$ form an exhaustion of $M$, and each $Z_i$ is forced to lie outside $Z_{i-1}$, either they form an exhaustion or they all intersect some compact set $U$. In the second case, for any $r>0$, we get a $Z_i$ for a large enough $i$ such that there is a stable minimal hypersurface separating the core from the end such that it intersects $U$ and has points at distance $r$ from $U$. This is case $(1)$ of the theorem.
    
    Alternatively, we get an infinite sequence of minimal hypersurfaces $Z_i$ such that they eventually escape every $U_k$. Then the sequence $\{V_i\}$ is an exhaustion of $M$ with minimal boundary. Since $M$ has positive scalar curvature, we could then apply Theorem \ref{theorempscexhaustion} to get a uniformly positive scalar curvature metric on $M$, contradicting our assumption. Therefore, we cannot get infinitely many minimal hypersurfaces by this process, which means that there exists some $Z_i$ (possibly empty) such that each $U_j\setminus V_i$ admits a singular strictly mean convex foliation for all $j\geq i$. Denoting this $V_i$ as $U$, the conclusion follows.

    Now assume that $M$ has an exhaustion $\{U_i\}$ with mean concave boundary. We again start by applying the mean curvature flow starting at $\partial U_0$. However, since $\partial U_0$ is mean concave with respect to the normal vector field pointing inward towards the interior of $U_0$, it has positive mean curvature with respect to the normal vector field pointing outward from $U_0$. Hence, the hypersurface $\partial U_0$ flows outward away from $U_0$ under the mean curvature flow. Like before, there are two possibilities. The mean curvature flow might converge to a minimal hypersurface $Z_0$ such that some components of $Z_0$ together separate $U_0$ from the end, and other components bound compact domains $K_{j'}$ that do not contain $Z_0$. Otherwise, we do not get any minimal hypersurface that separates $U_0$ from the end, and hence we get a singular strictly mean convex foliation of $M\setminus (U_0\cup_{j'}K_{j'})$ with respect to the outward pointing vector field from $U_0$. In the second case, it is actually a singular strictly mean concave foliation with respect to the inward pointing vector field, as is our convention. 

    Now, assume that we have obtained a singular strictly mean concave foliation of $M\setminus (U_0\cup_{j'}K_{j'})$. We start a new mean curvature flow at $\partial U_1$, and proceed as before. Again, we can either get a minimal hypersurface or a singular strictly mean concave foliation of $M\setminus U_1$. However, since the first iteration did not produce a minimal hypersurface $Z_0$, the flow starting at $\partial U_1$ cannot converge to a minimal hypersurface, since that would contradict the maximum principle. Therefore, getting a singular strictly mean concave foliation for $M\setminus U_0$ implies that we also get a singular strictly mean concave foliation for $M\setminus U_1$, and by the same argument, for each $M\setminus U_i$.

    On the other hand, if we had obtained a minimal hypersurface $Z_0$ in the first iteration that separates $U_0$ from the end, then we consider some $i'$ such that $Z_0$ lies in the interior of $U_{i'}$. We start the mean curvature flow at $\partial U_{i'}$ and repeat the argument. We either get a new minimal end separating hypersurface $Z_{i'}$, or a singular strictly mean concave foliation of $M\setminus (U_{i'} \cup_{j'}K_{j'})$, and hence of $M\setminus (U_j\cup_{j'}K_{j'})$ for all $j>i'$. If we get such a $Z_{i'}$, we can again choose $i''$ such that $Z_{i'}$ lies in the interior of $U_{i''}$ and then proceed as before. Repeating this process eventually results in some $U_{i_0}$ such that $M\setminus U_{i_0}$ has a singular strictly mean concave foliation, since otherwise we would get a nested infinite sequence of minimal hypersurfaces in $M$, which would give rise to an exhaustion with minimal boundary and lead to a contradiction as in the mean convex case. Note that here the end separating hypersurfaces cannot all intersect some compact domain, unlike the mean convex case. Then there exists a singular strictly mean concave foliation of $M\setminus (U_i\cup_{j'}K_{j'})$ for all $i\geq i_0$, and we are done.
\end{proof}

The above corollaries also hold in dimension $3$, where a stable minimal hypersurface is forced to be a sphere, and we are reduced to the case of connected sums. In the above argument, the mean curvature flow can of course be started at some arbitrary mean convex or mean concave hypersurface. Take an arbitrary set $N\subset M$ with mean convex (respectively mean concave) boundary. The same process then tells us that there exists a large enough $U\subset M$ such that if $U\subset N$ then $N\setminus U$ has a singular strictly mean convex foliation away from some compact domains (respectively, $M\setminus N$ has a singular strictly mean concave foliation). The existence of an exhaustion with a mean convex or a mean concave boundary is assumed to produce exactly an increasing sequence of such sets $N$.

Our final application shows that our constructions of manifolds with uniformly positive scalar curvature metrics can be slightly modified to obtain manifolds with positive scalar curvature, bounded geometry, and a desired rate of volume growth. Before stating the theorem, let us explicitly define what it means for a manifold to have the desired rate of volume growth. We refer to \cite{BM} and \cite{GP} for detailed discussions on controlling growth types of Riemannian manifolds.

\begin{defn}
    A function $v:\mathbb{N} \to \mathbb{R}^+$ is said to have bounded growth of derivative if there exists a positive integer $L$ such that, $\forall n\in \mathbb{N}$,
    $$\frac{1}{L}\leq v(n+2)-v(n+1)\leq L(v(n+1)-v(n)).$$ 
\end{defn}

\begin{defn}
    Two non-decreasing functions $f,h:\mathbb{N}\rightarrow \mathbb{R}_{+}$ are said to be of the same growth type if there exists an integer $A\geq 1$ such that for all $n\in \mathbb{N}$, 
\Bea
f(n)\leq Ah(An+A)+A \ \ \text{and} \ \ h(n)\leq Af(An+A)+A. 
\Eea
\end{defn}

\begin{corollary}
    Let $M$ be an open orientable manifold of dimension $4\leq n \leq 7$ satisfying the conditions of Theorem \ref{theoremmorsepsc}, Theorem \ref{theorempscexhaustion}, or part $
    (2)$ of Conjecture \ref{gromovconjecture} with $C>\frac{n-1}{n}$. Let $v$ be any function with bounded growth of derivative.
    \begin{enumerate}
     
     \item If $M$ has finitely many ends, then $M$ admits a complete Riemannian metric of bounded geometry with positive scalar curvature such that the volume growth function has the same growth type as $v$. 
    
     \item If $M$ has infinitely many ends, then $M$ admits a complete Riemannian metric of bounded geometry with positive scalar curvature such that the volume growth function has the same growth type as $v$ if and only if $\lim_{n\to \infty}\frac{v(n)}{n}=\infty$.
     \end{enumerate}
    
\end{corollary}

\begin{proof}
    In all three cases, by our constructions and the constructions in \cite{conjproof}, we get that the manifold admits a positive scalar curvature metric with cylindrical necks. In \cite{GP}, Grimaldi and Pansu constructed a metric with volume growth of the same growth type as $v$ by decomposing the manifold into pieces with product boundary, and controlling the length of the product necks. In \cite{Ours} it is shown that such a metric can be constructed while maintaining the positive scalar curvature condition. Hence, the crucial ingredient needed to apply Theorem $1$ of \cite{GP} is the existence of an exhaustion $\{U_i\}$ of $M$ such that each $U_{i+1}\setminus \mathring{U_i}$ admits a positive scalar curvature metric which is product metric near the boundary, such that the corresponding boundary metrics are isometric. This is of course satisfied by the manifolds we consider by virtue of having product necks, and hence the result holds.
 \end{proof}

For dimension $3$, the corresponding result is already known from \cite{Ours}. If $M$ satisfies the conditions of Theorem \ref{theoremmorsepsc} then the metric can be chosen such that it has uniformly positive scalar curvature. This is because the product necks can be chosen such that they are isometric to each other. Additionally, the compact pieces $U_{i+1}\setminus \mathring{U_i}$ would be of finitely many diffeomorphism types. This is because in the case of manifolds with product ends, the pieces $U_{i+1}\setminus \mathring{U_i}$ are eventually diffeomorphic to the same compact product manifold. Hence, the metric can be chosen such that the sectional curvature, injectivity radius, and scalar curvature are all bounded, which implies that there exists a metric with bounded geometry and uniformly positive scalar curvature. Note that we cannot guarantee the existence of a UPSC metric with the desired growth type for manifolds having minimal boundary exhaustions as per Theorem \ref{theorempscexhaustion} or for the general manifolds as per Conjecture \ref{gromovconjecture}, since we have no control over the diffeomorphism types of the pieces in between the minimal hypersurfaces. Hence, it is generally not possible to choose positive scalar curvature metrics on all the pieces while keeping the sectional curvature bounded above, injectivity radius bounded below, and the scalar curvature bounded below, since decreasing the sectional curvature or increasing the injectivity radius by scaling would also decrease the scalar curvature.


\begin{thebibliography}{HD}
\baselineskip=17pt


\bibitem{BM} Badura, \emph{Prescribing growth type of complete Riemannian manifolds of bounded geometry}, Ann. Polon. Math. {\bf75} (2000),  no. 2, 167-175.


\bibitem{Balacheff}    Balacheff, Gil Moreno de Mora Sard\`a and
              Sabourau,
    \emph{Complete 3-manifolds of positive scalar curvature with
              quadratic decay}, Math. Ann. {\bf392} (2025),
    no. 3, 4361-4389.

\bibitem{Taming3mfds} Chang, Weinberger and Yu,
     \emph{Taming 3-manifolds using scalar curvature},
   Geom. Dedicata {\bf148} (2010),
    3-14.

\bibitem{Chen} Chen, \emph{Optimal decay constant for complete manifolds of positive scalar curvature with quadratic decay}, Math. Ann. {\bf395} (2026), no. 1.


\bibitem{ChodoshLiChao}  
    Chodosh and Li,
     \emph{Generalized soap bubbles and the topology of manifolds with positive scalar curvature},
   Ann. of Math. (2) {\bf199}
      (2024),
    no. 2,
     707-740.


\bibitem{CMM} Chodosh,  Maximo, and Mukherjee, \emph{Complete Riemannian 4-manifolds with uniformly positive scalar curvature}, (2024), URL https://arxiv.org/pdf/2407.05574.
      

\bibitem{Ours} Das and Maity, \emph{Volume growth functions of complete Riemannian manifolds with positive scalar curvature}, arXiv preprint arXiv:2410.04121 (2024).

\bibitem{GP} Grimaldi and  Pansu, \emph{Bounded geometry, growth and topology}, J. Math. Pures Appl. {\bf 95} (2011), 85-98.


\bibitem{Gromovmetricinequalities} Gromov,
    \emph{Metric inequalities with scalar curvature}, Geom. Funct. Anal. {\bf28}
      (2018),
    no. 3,
     645-726.

\bibitem{Gromov5aspherical} Gromov, \emph{No metrics with positive scalar curvatures on aspherical 5-manifolds}, arXiv preprint arXiv:2009.05332 (2020).


\bibitem{FourLectures}  
   Gromov,
     \emph{Four lectures on scalar curvature},
 {Perspectives in scalar curvature. {V}ol. 1} (2023), 1-514.



\bibitem{GL} Gromov and  Lawson, \emph{The Classification of Simply Connected Manifolds of Positive Scalar Curvature}, Ann. Math. {\bf 111} (1980). 

\bibitem{GromovBlaine} Gromov and Lawson, Jr., 
     \emph{Positive scalar curvature and the {D}irac operator on complete
              {R}iemannian manifolds},
 Inst. Hautes \'Etudes Sci. Publ. Math.
 {\bf58}
      (1983),
    83-196.
      


\bibitem{Orikasa} Orikasa, \emph{Linking at Infinity and Scalar Curvature Decay on Non-Compact Manifolds}, arXiv preprint arXiv:2604.06547 (2026).

\bibitem{Rade} R{\"a}de, \emph{Scalar and mean curvature comparison via {$\mu$}-bubbles}, Calc. Var. Partial Differential Equations
	{\bf 62}
	(2023), no. 7.

\bibitem{ConformalXu} Rosenberg, Ruberman,  and Xu, \emph{The Conformal Laplacian and Positive Scalar Curvature Metrics on Manifolds with Boundary}, arXiv preprint arXiv:2302.05521
  (2023).


\bibitem{SY} Schoen and  Yau, \emph{On the structure of manifolds with positive scalar curvature}, Manuscripta Math. {\bf 28} (1979), no. 1-3, 159-183.


\bibitem{SYS}   Shi, Wang, Wu, and Zhu,
     \emph{On open manifolds admitting no complete metric with positive scalar curvature}, Ann. Inst. Fourier (Grenoble)
     (2026).

\bibitem{Song} Song, 
     \emph{A dichotomy for minimal hypersurfaces in manifolds thick at
              infinity},
   Ann. Sci. \'Ec. Norm. Sup\'er. (4)
 {\bf56}
      (2023),
    no. 4, 1085-1134.

\bibitem{Sweeney} Sweeney Jr., \emph{Positive curvature conditions on contractible manifolds}, 
  Math. Ann. {\bf394} (2026),
  no. 3.

\bibitem{JW} Wang, \emph{Topology of $3$-manifolds with uniformly positive scalar curvature}, arXiv preprint arXiv:2212.14383 (2023).

\bibitem{WangWhitehead}  Wang,
     \emph{Contractible 3-manifolds and positive scalar curvature ({I})},
   J. Differential Geom. {\textbf127} (2024),
    no. 3, 1267-1304.

\bibitem{conjproof} 
    Wang and Zhu, \emph{Quadratic Scalar Curvature Decay and Uniform Positivity in Dimensions Four through Seven}, arXiv preprint arXiv:2609.03676 (2026).

\bibitem{MWalsh} Walsh, \emph{Metrics of positive scalar curvature and generalised {M}orse functions, {P}art {I}}, Mem. Amer. Math. Soc. {\bf 209} (2011), no. 983, xviii+80.
     
\bibitem{WXZ} Wei, Xu and Zhang, \emph{Volume growth and positive scalar curvature}, Trans. Amer. Math. Soc. {\bf 378} (2025), no. 9, 6109-6136.
     

\bibitem{muExistence} Zhu,
     \emph{Width estimate and doubly warped product}, Trans. Amer. Math. Soc. {\bf374} (2021),
    no. 2,
     1497-1511.


       


\end{thebibliography}
\end{document}